\documentclass[12pt,index]{amsart}

\usepackage{amssymb}
\usepackage{mathtools}

\usepackage{enumerate,xcolor,mathrsfs}

\usepackage[pagebackref,colorlinks,linkcolor=red,citecolor=blue,urlcolor=blue,hypertexnames=true]{hyperref}
\newcommand{\df}[1]{{\it{#1}}{\index{#1}}}

\makeatletter
\newsavebox\myboxA
\newsavebox\myboxB
\newlength\mylenA

\newcommand*\xoverline[2][0.75]{%
    \sbox{\myboxA}{$\m@th#2$}%
    \setbox\myboxB\null% Phantom box
    \ht\myboxB=\ht\myboxA%
    \dp\myboxB=\dp\myboxA%
    \wd\myboxB=#1\wd\myboxA% Scale phantom
    \sbox\myboxB{$\m@th\overline{\copy\myboxB}$}%  Overlined phantom
    \setlength\mylenA{\the\wd\myboxA}%   calc width diff
    \addtolength\mylenA{-\the\wd\myboxB}%
    \ifdim\wd\myboxB<\wd\myboxA%
       \rlap{\hskip 0.5\mylenA\usebox\myboxB}{\usebox\myboxA}%
    \else
        \hskip -0.5\mylenA\rlap{\usebox\myboxA}{\hskip 0.5\mylenA\usebox\myboxB}%
    \fi}
\makeatother

\newtheorem{theorem}            {Theorem}[section]
\newtheorem{corollary}          [theorem]{Corollary}
\newtheorem{proposition}        [theorem]{Proposition}

\newtheorem{lemma}              [theorem]{Lemma}
\newtheorem{remark}         [theorem]{Remark}

\newcommand{\CC}{\mathbb{C}}
\newcommand{\RR}{\mathbb{R}}

\renewcommand{\AA}{\mathbb{A}}

\newcommand{\cF}{\mathcal{F}}

\newcommand{\DD}{\mathbb{D}}

\newcommand{\ZZ}{\mathbb{Z}}

\newcommand{\QA}{\mathbb{Q}\mathbb{A}}
\newcommand{\sQA}{\mathscr{Q}\mathbb{A}}
\newcommand{\TT}{\mathbb{T}}

\newcommand{\PA}{\mathbb{P}\mathbb{A}}
\newcommand{\sPA}{\mathscr{P}\mathbb{A}}

\newcommand{\fH}{\mathfrak{H}}
\newcommand{\fT}{\mathfrak{t}}
\newcommand{\fD}{\mathfrak{d}}
\newcommand{\fA}{\mathfrak{A}}
\newcommand{\fP}{\mathfrak{P}}
\newcommand{\vg}{{\tt{g}}}
\newcommand{\SSA}{\mathbb{S}\mathbb{A}}
\newcommand{\sSSA}{\mathscr{S}\mathbb{A}}

\newcommand{\FF}{\mathbb{F}_g^2}

\newcommand{\sA}{\mathscr{A}}
\newcommand{\ow}{\overline{w}}

\newcommand{\setF}{\mathfrak{F}}
\newcommand{\rr}{\psi}
\newcommand{\aaa}{\alpha}
\newcommand{\bb}{\beta}

\newcommand{\iE}{\iota}
\newcommand{\tpi}{\vartheta}
\newcommand{\GE}{G}
\DeclareMathSymbol{\shortminus}{\mathbin}{AMSa}{"39}
\newcommand{\Cxix}{\mathbb{C}[x,x^{\shortminus 1}]}
\makeatletter
\newcommand{\addresseshere}{%
  \enddoc@text\let\enddoc@text\relax
}
\makeatother

\title[Geometric Dilations]{Geometric Dilations and Operator Annuli}

\author[McCullough]{Scott McCullough}
\address{Department of Mathematics\\
 University of Flordia \\ Gainesville, FL}
\email{sam@ufl.edu}

\author[Pascoe]{James E. Pascoe${}^1$}
\address{Department of Mathematics\\
 University of Flordia \\ Gainesville, FL}
\email{pascoe@ufl.edu} 
\thanks{Research Supported by NSF grant DMS-1953963}

\subjclass[2010]{47A20 (Primary); 47B20, 47B91 (Secondary)}
\keywords{quantum annulus, complete Pick kernel, operator extension, dilation theory}

 \numberwithin{equation}{section}

\makeindex

\begin{document}

\begin{abstract}
 Fix $0<r<1.$ The dilation theory for the quantum annulus $\QA_r,$ consisting of those
 invertible Hilbert space operators $T$  satisfying $\|T\|\,,\|T^{-1}\|\le r^{-1},$
 is determined. The proof technique involves a
 geometric approach to dilation that applies to
 other well known dilation theorems. The dilation theory for the quantum
 annulus is compared, and contrasted, with the dilation theory for
 other canonical operator annuli. 
\end{abstract}

\maketitle

\section{Introduction} 
  Following  M. Mittal in \cite{Mittal} (see 
 also the references therein; e.g., \cite{VernK,Shields}), given $0<r<1,$
  we call \df{$\QA_r,$}
 the collection of invertible  Hilbert space operators $T$
 satisfying $\|T\|\,,\|T^{-1}\|\le r^{-1},$ 
 the \df{quantum annulus}.   In this article we determine
 the dilation theory for $\QA_r.$ The result is stated
 in  Subsection~\ref{s:QAint} immediately below.  The key proof technique 
 -- an adaptation of Nelson's trick \cite{Nelson} --
 is of independent interest. We illustrate how the trick provides
 a geometric alternate approach to the Sz.-Nagy dilation theorem
 and related results. See Section~\ref{s:Nelson}.
 Later in this introduction we discuss connections between the quantum annulus 
 and  other  operator annuli  \cite{BY,CG,TsQA,TsPA}.

\subsection{The quantum annulus}
\label{s:QAint}
For a Hilbert space $H,$ let \df{$B(H)$}
 denote the bounded operators on $H.$
 An invertible operator  $T\in B(H)$ has a \df{$\Cxix$-dilation}
 to an invertible operator  $J\in B(K)$  if there is
 an isometry $V:H\to K$ such that
\begin{equation}
 \label{e:JdilateT}
 T^n =V^* J^n V,
\end{equation}
 for all integers $n.$
  For expository ease, we will often drop the $\Cxix$
  modifier.

 It is evident that if $T$ dilates to  $J$  
 and $J$ is in the quantum annulus, then 
  $T$ is also in the quantum annulus; that is, 
 $\QA_r$ is \df{closed with respect $\Cxix$-compressions}. 
 Hence the collection
 $\QA_r$ is, what we call here, a \df{$\Cxix$-dilation family} 
 \cite{Amodel} in the sense
 that it is closed with respect to (a)  direct sums,
 (b) unital $*$-representations and (c) $\Cxix$-compressions.

 The dilation of equation~\eqref{e:JdilateT}
 is \emph{trivial} \index{trivial dilation} if the range of $V$ reduces $T.$
 An element $T\in \QA_r$ is 
 \df{dilation extremal in $\QA_r$} if
 all of its dilations to $\QA_r$  are trivial; that is if $J\in \QA_r$ 
 and equation~\eqref{e:JdilateT} holds, then the range of $V$ reduces $J.$
 This notion of dilation extremal has connections
 with boundary representations in the sense of Arveson. See
 \cite{Arveson,DK,DM,MS} and the references therein. 

 Let \df{$\sQA_r$} denote those operators $J$ that,  up to unitary equivalence, 
 take  the form
\[
 J=U \begin{pmatrix} r I_{K_{+1}} & 0\\0 & r^{-1} I_{K_{-1}}\end{pmatrix},
\]
 where $U$ is unitary, $J$ acts on the Hilbert 
 space direct sum $K_{+1} \oplus K_{-1}$
 and $I_{K_{\pm 1}}$ is the identity on $K_{\pm 1}.$
 It is immediate that
 $J\in \QA_r.$  As an alternate description,
  $J\in \sQA_r$ if and only if there
 exists projection $P_{\pm}$ that sum to the identity 
 such that $J^*J=r^2P_+\, + \, r^{-2}P_-.$
  Theorem~\ref{t:QA} below, proved in 
 Section~\ref{s:QA}, describes the dilation theory of
 $\QA_r,$ where $\QA_r[H]$ denotes those elements of $\QA_r$
 that act on the Hilbert space $H.$

\begin{theorem}
\label{t:QA} The collection $\sQA_r$ has the following
properties.
\begin{enumerate}[(a)]\itemsep=5pt
\item \label{i:Qextremal}
   If $J\in \sQA_r,$ then $J\in \QA_r$ is 
   dilation extremal in $\QA_r$; 
\item \label{i:Qrep-closed}
   If $J\in \sQA_r[K]$ 
 and $\pi:B(K)\to B(\GE)$ is a unital $*$-representation on
 a Hilbert space $\GE,$ then $\pi(J)\in \sQA_r[\GE].$
\item \label{i:Qdilate}  If $T\in \QA_r,$ then $T$ 
 dilates to a $J\in \sQA_r.$
\end{enumerate}
\end{theorem}

 Thus $\sQA_r$ is closed with respect to arbitrary direct sums,
 unital $*$-representations and restrictions to reducing subspaces
 and each element of $\QA_r$ dilates to an element of $\sQA_r.$
 Moreover, item~\eqref{i:Qextremal}  says that $\sQA_r$ is
 the smallest subcollection of $\QA_r$ with these properties.
 Hence $\sQA_r$ 
 is canonical and  deserves the moniker \df{dilation boundary} of $\QA_r.$
 (See \cite{Amodel}.)
  The proof of item~\eqref{i:Qdilate}, 
 the dilation,  uses Nelson's trick to produce a geometric dilation.
 See Section~\ref{s:QA}.

\subsection{The Pick annulus}
  Let \df{$\AA_r$} denote the annulus,
\[
  \AA_r =\{z\in \CC: r<|z|<\frac1r\}.
\] 
 The recent papers \cite{BY,TsPA} consider the family of operators
 associated to the kernel $k_r:\AA_r\times \AA_r\to\CC$ \index{$k_r$}
 defined by
\[
k_r(z,w) = \frac{r^{-2}-r^2}{r^{-2}+r^2-zw^*- (zw^*)^{-1}}.
\]
 The Agler model theory for this family was determined in \cite{BY}.
 In \cite{TsPA} the spectral bound for $\PA_r$ is determined, with
 the lower bound also obtained in \cite{BY}. (Tsikalas also obtains
 a lower bound for the spectral constant for the quantum annulus in \cite{TsPA}.) See
 Section~\ref{s:inner}.

 For a self-adjoint operator $A$ on Hilbert space, 
 $A\succeq 0$ \index{$\succeq0$}
 indicates $A$ is positive semidefinite.
 Let \df{$\PA_r$} denote those invertible
 Hilbert space operators $T$ satisfying 
 $\frac{1}{k_r}(T,T^*)\succeq 0$ in the \df{hereditary} sense that
\begin{equation}
\label{d:krT}
\frac{1}{k_r}(T,T^*)=\frac{1}{r^{-2}-r^2} \left (r^{-2}+r^2 -T^*T - T^{-*}T^{-1}\right ) \succeq 0.
\end{equation}
While not obvious, if $T\in\PA_r[H],$ then 
$T$ is in $\QA_r[H].$ See \cite[Proposition~2.2]{BY} and \cite[Lemma 4.1]{TsPA}.

 Because the 
 collection $\PA_r$ has an hereditary
 definition (all the adjoints are on the left in 
 equation~\eqref{d:krT}), it is closed with respect
 to restrictions to $\Cxix$-invariant subspaces:
 if $Y\in \PA_r[K]$ and $H\subseteq K$ is a subspace
 invariant for both $Y$ and $Y^{-1},$ then the 
 restriction of $Y$ to $H$ is in $\PA_r[H].$
 Thus $\PA_r$ is a \df{family}, \index{Agler family} in the sense of
 Agler \cite{Amodel}, with respect to $\Cxix:$
   $\PA_r$ is closed with respect to (a) direct sums, (b)
  unital $*$-representations 
 and (c) restrictions to $\Cxix$-invariant subspaces.  

  An operator $T\in B(H)$ \df{lifts} to an operator $J\in B(K)$
 if there is an isometry $V:H\to K$ such that 
 $VT=JV.$ Equivalently, $T$ is, up to unitary equivalence,
 the restriction of $J$ to an invariant subspace. 
 In the case that both $J$ and $T$ are invertible, it follows
 that $VT^n=J^nV$ for all integers $n.$ In particular,
  $T^n =V^*J^nV$ and hence $T$ dilates to $J.$ 

 Let  \df{$\sPA_r$} denote the collection of Hilbert space operators 
 that have, up to unitary equivalence,  the form
\begin{equation*}
%\label{e:J}
  J = \begin{pmatrix} rV & \frac1r E \\ 0 &\frac1r W^* \end{pmatrix}
 = \begin{pmatrix} V & E \\ 0 & W^* \end{pmatrix} \, \begin{pmatrix}   
   rI_{K_{+1}}&0\\0&r^{-1} I_{K_{-1}}\end{pmatrix},
\end{equation*}
where 
\[
 U=  \begin{pmatrix} V  & E \\0 & W^*\end{pmatrix},
\]
 is a unitary matrix. In particular, $V$ and $W$ are isometries
 and $E$ is a partial isometry with initial space $\ker W^*$
 and final space $\ker V^*.$ 
  A calculation shows
 $\frac{1}{k_r}(J,J^*)=-\frac{1}{k_r}(J^*,J)$ is 
 the projection onto  $\ker V^*.$
 Thus elements in $\sPA_r$ are distinguished members of $\PA_r.$

 The following lifting theorem, which essentially identifies 
 the Agler model theory for $\PA_r$  was established 
 in \cite{BY}.

\begin{theorem}[{\cite{BY}}]
\label{t:BY}
 If $T\in B(H)$ is invertible, then $T$ is in $\PA_r[H]$ if and only if $T$ lifts
 to a $J$ in $\sPA_r.$
\end{theorem}

  The \df{Agler boundary} \cite{Amodel} of $\PA_r$
  is the smallest subcollection $\partial \PA_r$
  of $\PA_r$
  that is (1) closed with respect to (a) direct sums, 
 (b) unital *-representations, and (c) restrictions
 to reducing subspaces; and (2) 
 each $T$ in $\PA_r$ lifts to a $J$ in $\partial \PA_r.$
 While it is not evident, families have 
 such a boundary \cite{Amodel}.
  For the family $\PA_r,$ 
  it is straightforward to verify that  $\sPA_r$ 
 satisfies the conditions (1).  On the other hand,
 as observed in Section~\ref{s:PA},
 each $J\in \sPA_r$ is \df{extremal} -- if $J$
 lifts to $F\in \PA_r$ as $VJ=JV,$ then
 the range of $V$ reduces $F$ -- 
% in the sense that
% any lift of $J$ to an $F\in \PA_r$ as
% $VJ=FV$  is \df{trivial}, in the sense that the
% range of $V$ reduces $F,$ 
 and is therefore in $\partial \PA_r.$ 
  Thus Theorem~\ref{t:BY}
 has the following corollary.

\begin{corollary}
 \label{c:PA} 
 The collection $\sPA_r$ is the Agler boundary of $\PA_r.$
\end{corollary}

 Since the kernel $k_r$ is a
 complete Pick kernel \cite{AM2000}\cite{TsPA}
(see also \cite{AHMR19,AHMR21,Hartz} for recent results
 for complete Pick kernels), 
  the family $\PA_r$ is also 
 a dilation family. 
 Indeed, if $T\in B(H)$ is invertible and there is
 a $Y\in \PA_r[K]$ and an isometry $V:H\to K$ such that
 $T^n = V^*Y^nV$ for all integers $n,$ 
then\footnote{It is enough that $T^n=V^*J^nV$ for $n=\pm 1.$}
\[
\begin{split}
 (r^{-2}-r^2)\, \frac{1}{k_r}(T,T^*) & =  r^2+r^{-2} -T^*T- T^{-*}T^{-1}\\
 &=  r^2+r^{-2} -V^*J^*V V^*JV - V^* J^{-*}V V^* J^{-1}V \\
 & \succeq r^2+r^{-2} - V^*J^*JV -V^* J^{-*}J^{-1}V \\
 & = V^*(r^2+r^{-2} - J^*J-J^{-*}J^{-1})V\\
 &=(r^{-2}-r^2)\, V^*\frac{1}{k_r}(Y,Y^*)V\succeq 0,
 \end{split}
\]
 and hence $T\in \PA_r[H].$
 A consequence, stated precisely in Corollary~\ref{c:BY}
 below, is that the dilation theory and the Agler model
 theory for $\PA_r$ are essentially the same.

\begin{corollary}
\label{c:BY}
 For an invertible operator $T\in B(H),$ the following are equivalent.
 \begin{enumerate}[(i)]\itemsep=5pt
 \item $T\in \PA_r;$
  \item  $T$  dilates to a $Y\in \PA_r;$
 \item $T$ lifts  to a $J\in \sPA_r.$
\end{enumerate}

 Moreover, 
\begin{enumerate}[(a)]\itemsep=5pt
\item \label{i:Pextremal}
   if $J\in \sPA_r,$ then $J$ is dilation extremal; that is,
 any dilation of $J$ to an element of $\PA_r$ is trivial;
\item \label{i:Prep-closed}
  The collection $\sPA_r$ is closed with respect to
 arbitrary direct sums, unital $*$-representations and
 restrictions to reducing subspaces.
\end{enumerate}
\end{corollary}

\subsection{The annulus as a spectral set}
 Let \df{$R(\AA_r)$} denote the algebra of rational functions with poles off 
 $\overline{\AA_r}.$  An operator $T$ with spectrum in $\overline{\AA_r}$
 has $\AA_r$ as \df{spectral set} if,  for each $\rr\in R(\AA_r),$
\begin{equation*}
%\label{e:specset}
 \|\rr(T)\| \le \|\rr\|_\infty,
\end{equation*}
where \df{$\|\rr\|_\infty$} is the sup norm 
 of $\rr$ on $\overline{\AA_r}.$ 

Let
 \df{$\SSA_r$} denote the  collection of
   operators with $\AA_r$ as
 a spectral set. Let  \df{$\sSSA_r$} denote the normal operators
 with  spectrum in the boundary of $\AA_r.$
 In  \cite{Agler-annulus} Agler proves the following
 analog of the Sz.-Nagy dilation theorem for $\AA_r.$
 If $T\in \SSA_r[H]$ and $\sigma(T)\subseteq \AA_r,$  then 
 there exists a Hilbert space $K$ 
 an operator $N\in \sSSA_r[K]$ 
 and an isometry $V:H\to K$ such  that 
\[
 \psi(T)=V^*\psi(N)V,
\]
 for all $\psi\in R(\AA_r).$

 The collection \df{$\SSA_r$} is both a family in the sense of Agler
 and a dilation family. 
 The statement of Agler's dilation theorem 
 essentially identifies the normal operators with the
 spectrum in the boundary of $\AA_r$ as the dilation boundary
 of $\SSA_r.$

 A normal operator with spectrum in the boundary
 of $\PA_r$ has, up to unitary equivalence, the form
\[
  \begin{pmatrix} U_{+1} & 0 \\ 0 & U_{-1} \end{pmatrix}
 \, \begin{pmatrix} r & 0\\0 & r^{-1} \end{pmatrix},
\]
 where $U_{\pm 1}$ are unitary operators. 
 As noted in \cite{BY,TsPA}, there are the  proper inclusions
\[
 \SSA_r \subsetneq \PA_r\subsetneq \QA_r.
\]
 Evidently from the discussions
 above, but somewhat  surprisingly,
  the same chain of inclusions holds
 for the dilation boundaries; that is,
\[
 \sSSA_r \subsetneq \sPA_r\subsetneq \sQA_r.
\]

\begin{remark}\rm
 The family $\SSA_r$ is naturally identified with unital contractive
 (equivalently unital completely contractive) 
 representations of $R(\AA_r)$ on Hilbert space. Namely,
 an operator $T\in \SSA_r[H]$ determines the representation
 $\pi_T:R(\AA_r)\to B(H)$ given by $\pi_T(\rr)=\rr(T).$
 Both $\QA_r$ and $\PA_r$ are also naturally identified
 with collections of representations of $\Cxix$ (or $R(\AA_r)$).
\end{remark}

\subsection{Reader's guide}
  Section~\ref{s:Nelson}
 contains geometric arguments in favor of
 the Sz.-Nagy dilation theorem,  its
 multivariable non-commutative analog
 for row contractions and the dilation 
 theory for doubly commuting contractions 
 all using Nelson's trick. The trick is used
 in Section~\ref{s:QA} to prove Theorem~\ref{t:QA}.

 Section~\ref{s:PA}  provides some details for the 
 proofs of Corollaries~\ref{c:PA} and \ref{c:BY}.
 We emphasize the key ingredient is Theorem~\ref{t:BY}.

 The paper concludes with a few remarks about 
 spectral constants for $\QA_r$ 
 and  musing about Nelson's trick and 
 Ando's Theorem in Section~\ref{s:musing}

\section{Automorphic Nelson's Trick}
\label{s:Nelson}
 To prove  the principal part of Theorem~\ref{t:QA},
 item~\eqref{i:Qdilate}, we adapt 
 a method found in \cite{Nelson} that replaces
 the positive definite (diagonal) matrix in the singular
 value decomposition of a square matrix with automorphisms
 of the unit disc.  Similar methods were employed by 
 \cite{Hartz-iumj} for multivariable weighted shifts. 
 Hartz notes  the trick also appears in \cite{Paulsen} 
 as well as in \cite{Pisier}.  In this section 
we  illustrate the technique by indicating how it
 applies to several familiar dilation results. Namely
 the Sz.-Nagy dilation theorem (Subsection~\ref{s:Nagy}),
    the Bunce-Frazho-Popescu multivariable noncommutative 
 theory of \df{row contractions} \cite{B,F82,F84,Po89,Po91}
 (Subsection~\ref{s:row}) and doubly commuting
 contractions (Subsection~\ref{s:double}).
 In Section~\ref{s:QA} we apply the method to the quantum
 annulus.

\subsection{Unitary dilations of contractions}
\label{s:Nagy}
 For a square matrix $T$ of size $n$ with $\|T\|<1,$ 
 Nelson's construction proceeds as follows.
 The singular  value decomposition of $T$ has the form
 $T=UDW,$ where $D$ is a diagonal matrix whose diagonal entries $\lambda_j$
  are the eigenvalues of $(T^*T)^{\frac12}$
 and $U,W$ are unitary matrices.  In particular
 $0\le \lambda_j <1.$
 Let
\[
 b_j(z)= \frac{\lambda_j-z}{1-\lambda_j z}.
\]
 Thus each $b_j$ is an automorphism of the unit disc $\DD=\{z\in \CC: |z|<1\}.$  Let  \index{$\DD$} \index{unit disc}
\begin{equation}
 \label{e:D}
 D(z) =\begin{pmatrix} b_1(z) & 0& 0 &\cdots & 0\\ 0 & b_2(z) &0 &\cdots & 0\\
   \vdots & \vdots &\vdots & \cdots &  \vdots \\ 
    0&0& 0& \cdots & b_n(z)\end{pmatrix}
\end{equation}
 and note that $D(0)=D.$ Let $F(z)= UD(z)W.$ Thus $F:\DD\to M_n(\CC)$ 
 is a bounded analytic function such that $F(0)=T$ and  $F(\zeta)^*F(\zeta)=I_n$
 for $|\zeta|=1.$  
We note that our construct slightly refines Nelson's in that originally he used independent variables on the diagonal,
which, as we now turn to dilation theory, is less desirable from a minimality perspective.

  We now show how to produce the desired dilation
 of $T.$  Let \df{$H^2(\DD)$} denote the usual
 Hardy-Hilbert space.  By construction, the operator $M_F$ of multiplication
 by $F$ on $H^2(\DD)\otimes \CC^n$  is an isometry. (Note that the cokernel has dimension at most $n.$)  The mapping
 $V:\CC^n\to H^2(\DD)\otimes \CC^n$ defined by $Vx= 1\otimes x$ is an isometry and
\[
 M_F^*Vx= M_F^* 1\otimes x = 1\otimes F(0)^* x = 1\otimes T^*x =VT^*x.
\]
Hence $T^*$ \df{lifts} 
 to the  coisometry $M_F^*,$ providing an explicit, 
geometric,  unitary dilation of $T.$

\subsection{Row contractions}
\label{s:row}
  The Nelson argument extends to 
  the Bunce-Frazho-Popescu multivariable noncommutative 
 theory of \df{row contractions} \cite{B,F82,F84,Po89,Po91}.
 Given a tuple $T=(T_1,\dots,T_\vg)$ of $n\times n$ matrices,
 view $T$ as the $n \times n\vg$ matrix
\[
 T=\begin{pmatrix} T_1 & \cdots & T_\vg\end{pmatrix} 
   \in M_{n,n \vg}(\CC).
\]
  Suppose $T$ is a strict contraction;  that is $\|T\|<1.$

 Once again, consider the singular value decomposition 
\[
 T= UEW,
\]
 where 
\[
 E=\begin{pmatrix} D & 0 & \dots & 0 \end{pmatrix} 
\]
 for a diagonal $n\times n$ matrix  $D$ with nonnegative
 entries that are strictly less than one. In particular, 
 the unitary matrices $U$ and $W$ have sizes $n$ and $n\vg$
 respectively. 
 Define $D(z)$ is as in equation~\eqref{e:D}.
 Thus $D(0)=D$ and $D$ is unitary valued on the boundary 
 of the disc.

 Write $W$ as a block $\vg\times\vg$  matrix
 with entries $W_{j,k}\in M_n(\CC)$ and define
 $F_j:\DD\to M_n(\CC)$ by 
 $F_j(z)= U D(z) W_{1,j}.$
 Let $M_j$ denote the operator of multiplication
 by $F_j$ on $H^2(\DD)\otimes \CC^n$ and observe
\begin{equation}
\label{e:rows}
 M_j^* 1\otimes x = 1\otimes F_j(0)^*x 
   = 1\otimes T_j^*x.
\end{equation}

 Let $\FF$ denote the freely noncommutative Fock space
 on $\vg$ letters $x=\{x_1,\cdots,x_\vg\}$ and 
 let $\varnothing$ denote the empty word (vacuum state). Thus
 $\FF$ is the Hilbert space with orthonormal basis
 the words in $x.$ Let $S_j:\FF\to\FF$ denote the shift
 determined by $S_jf =x_jf$ for $f$ a finite $\CC$-linear
 combination of words and note that $S_j^*S_k=\delta_{j,k}I;$
 that is, the $S_j$ are isometries with pairwise orthogonal ranges.

 Define $\iE:H^2(\DD)\to \FF\otimes H^2(\DD)$ by 
 $\iE h=\varnothing \otimes h.$
 Thus $\iE$ is an isometry. 
 On the Hilbert  space 
\[
K=\left[ H^2(\DD) \otimes \CC^n  \right ] \, \oplus \, \oplus_{2}^{\vg} 
  \left [\FF \otimes H^2(\DD) \otimes \CC^n  \right],
\]
 let
\[
 J_j  = \begin{pmatrix} M_j  & 0 \\
 \begin{bmatrix} \iE\otimes UW_{2,j} \\ \vdots \\ \iE\otimes UW_{\vg,j} \end{bmatrix}
  &  I \otimes S_j \otimes I \end{pmatrix}.
\]
 For instance, with $\vg=3,$
\[
 J_j =  \begin{pmatrix} M_j & 0 &0 \\ 
    \iE\otimes UW_{2,j}   & S_j\otimes I &0 \\
    \iE\otimes U^*W_{3,j}  & 0 & S_j\otimes I \end{pmatrix}.
\]
 Observe that $J_j^*J_k =\delta_{j,k}I.$ Thus $\{J_1,\dots,J_\vg\}$
 is a family of isometries with orthogonal ranges. 

 Define $V:\CC^n\to K$ by $Vx=[1\otimes x] \oplus 0.$ Thus
 $V$ is an isometry and, in view of  equation~\eqref{e:rows},
\[
 J_j^*V= VT_j^*. 
\]

\subsection{Doubly commuting contractions via Nelson}
\label{s:double}
 We note that Nelson's trick also works for a tuple
 of \df{doubly commuting} contractions.  See \cite{BNF, BNS}
 and the references therein.
  Suppose
 $T_1,\dots,T_d$ are commuting $n\times n$ 
 matrices that are strict contractions. 
 Suppose further they  are invertible and 
 \df{doubly commute}, meaning for each $j\ne k,$
\[
 T_j^*T_k= T_k T_j^*.
\]
 Each $T_j$ has its polar decomposition,  $T_j=U_j D_j,$
where $D_j=(T_j^*T_j)^{\frac12}$ and $U_j = T_j D_j^{-1}.$
Thus each $D_j$ is a positive definite strict contraction
  and each $U_j$ is unitary. 
The doubly commuting hypothesis implies  the $D_j$
 commute with one another, the $U_j$ doubly commute
 and $D_jU_k=U_kD_j$ for $j\ne k.$  
 Since the $D_j$ are 
   commuting self-adjoint matrices, they 
 are simultaneously diagonalizable. Hence, for each 
 $1\le j \le d,$ there are pairwise orthogonal projections
 $P_{j,1},P_{j,2},\dots,P_{j,m_j}$ that sum to the identity and
 distinct $0< \lambda_{j,\aaa} <1$ such that
\[
 D_j =\sum_{\aaa=1}^{m_j} \lambda_{j,\aaa} P_{j,\aaa}. 
\]
Further, for all $j,k,\aaa,\bb,$ the operators  $P_{j,\aaa}$ and $P_{k,\bb}$
 commute and, for $j\ne k,$ the operators
 $U_k$ and $P_{j,\bb}$ commute. Let
\[
 b_{j,\aaa}=  \frac{\lambda_{j,\aaa}-z}{1-\lambda_{j,\aaa}z}
\]
 and 
\[
 D_j(z) = \sum_{\aaa=1}^{m_j} b_{j,\aaa}(z) P_{j,\aaa}.
\]
 It follows that $D_j(z)$ and $U_k$ commute for $j\ne k,$
 and the $D_j(z)$ commute with one another. Hence  the resulting
 matrix functions $F_j(z)=U_jD_j(z)$ pointwise doubly commute
 and are unitary valued on the boundary of the disc. 
 Thus the  operators $M_{j}$ of multiplication
 by $F_j$ on $H^2(\DD)\otimes \CC^n$ 
 are isometries and doubly commute. 
 With the usual isometry  
 $V:\CC^n\to H^2(\DD)\otimes \CC^n$ defined by $Vh=1\otimes h,$ 
  one finds
 $VT_j^* = M_{j}^* V$ for each $j.$

\section{The Quantum Annulus}
\label{s:QA}
 In this section we prove Theorem~\ref{t:QA}.
 A preliminary version of the dilation, item~\eqref{i:Qdilate},
 is established using Nelson's trick in Subsection~\ref{s:Nelson-QA}.
 The remaining items of the theorem are proved in 
  Subsection~\ref{s:ArvQA} and are then used, in conjunction
 with the result of Subsection~\ref{s:Nelson-QA},
 to complete the proof of item~\eqref{i:Qdilate} in
 Subsection~\ref{s:proveQA}.

\subsection{Nelson's trick applied to $\QA_r$}
\label{s:Nelson-QA}
 In this subsection
 we  apply Nelson's trick to obtain Lemma~\ref{l:QA} below, an  initial version
 of item~\eqref{i:Qdilate} of Theorem~\ref{t:QA}.
  The proof of Theorem~\ref{t:QA}
 concludes in subsection~\ref{s:proveQA}.
 Let \df{$\sigma(T)$} denote the spectrum of a bounded operator $T$
 on Hilbert space.

\begin{lemma}
 \label{l:QA} 
   Suppose $T\in \QA_r[H].$
  If $\sigma((T^*T)^{\frac12})\subseteq (r,r^{-1})$  is finite, then
  there exists a Hilbert space $K,$ an operator $J\in \sQA_r[K]$
  and an isometry $V:H\to K$ such that
 $T^n=V^*J^nV$ for all integers $n.$
\end{lemma}

\begin{proof}
 The operator $T$ has the polar decomposition
\[
 T= UP,
\]
 where $P$ is positive semidefinite and $U$ is unitary. 
 Indeed, since $T^*T$ is invertible,  $P=(T^*T)^{\frac12}$
 and $U=TP^{-1}.$  By hypothesis, 
 there exists a finite set $\setF \subseteq (r,r^{-1})$ 
  such that $P$ has spectral decomposition,
\[
 P = \sum_{\lambda\in \setF} \lambda E_\lambda,
\]
 where $\{E_\lambda:\lambda\in \setF\}$ are pairwise
 orthogonal projections that sum to the identity. 
 
 Let  \df{$\TT$} denote the unit circle, viewed as the
 boundary of $\DD.$  Let $I^{\pm 1}$ denote
 the upper and lower half of $\TT$ respectively. 
 It is well known that the unit disc $\DD$ is the 
 universal cover of $\AA_r.$ 
 Namely, there exists  an onto analytic function $\tpi:\DD\to\AA_r$
 with $\tpi(0)=1$ and mapping $I^{\pm 1}$ to the 
 inner and outer boundary components 
 $\{|z|=r^{\pm 1}\}$ of $\AA_r$  respectively 
 that extends across $\TT$   except at $\pm 1.$ 
 Given $r<\lambda <r^{-1},$  there exists a M\"obuis automorphism
 $m_\lambda$ of the unit disc so that 
 $\tpi_\lambda= \tpi\circ m_\lambda$ maps $\DD$
 onto $\AA_r$ and sends $0$ to  $\lambda.$  There
 exists arcs $I_\lambda^{\pm 1}$ whose disjoint union is the 
 boundary of the unit disc, save for two points,
 that are mapped, under $\tpi_\lambda,$ to the inner and
 outer boundaries of $\AA_r$ respectively.

 Define $F:\DD\to B(H)$ by $F(z)=UP(z),$ where
\[
 P(z) =
  \sum_{\lambda\in \setF} \tpi_\lambda(z) E_\lambda.
\]
Observe $F$ is analytic on $\DD,$ extends to 
 $\TT$ except at finitely many point, 
 and $F(0)=T.$

 Let $L^2=L^2(\TT)$ denote the usual $L^2$ space of the unit circle
 $\TT.$  The characteristic functions $\chi_{\lambda}^{\pm 1}$ 
 of the intervals $I_\lambda^{\pm 1}$ induce  projection operators
 $Q_\lambda^{\pm 1}$ on $L^2(\TT)$ by sending $f$ to $\chi_{\lambda}^{\pm 1}f.$
 Moreover, the projections $Q_\lambda^{\pm 1}$
 are orthogonal and sum to the identity.  Let
 $K=L^2(\TT)\otimes H.$ It follows that
 the $2|\setF|$ projections 
 $\{Q_\lambda^{\pm 1} \otimes E_\lambda:  \lambda\in \setF \}$
 are pairwise orthogonal and sum to the identity.

 Let $M_F$  denote the operator of multiplication by $F$ on 
 $K=L^2(\TT)\otimes H.$
 By construction, given $f\in L^2(\TT),$ for  $h\in H$ 
 and $\zeta\in \partial \DD$ and each $\lambda\in \setF,$
\[
 M_F(\zeta) [Q_\lambda^{\pm 1} \otimes E_{\lambda}] (f\otimes h)(\zeta)
   =  M_F(\zeta)  \chi_{\lambda}^{\pm 1}f(\zeta) \otimes E_\lambda h 
= \tpi_\lambda(\zeta)\chi_{\lambda}^{\pm 1}f(\zeta) \otimes  UE_\lambda h.
\]
Hence, if also $f^\prime \in L^2(\TT)$ and $h^\prime \in h,$ then
\[
\begin{split}
  \langle M_F^* M_F & [Q_\lambda^{\pm 1} \otimes E_{\lambda}] (f\otimes h),
    f^\prime \otimes h^\prime \rangle \\
&  = \langle \,  \tpi_\lambda(\zeta)\chi_{\lambda}^{\pm 1}f(\zeta) \otimes  UE_\lambda h,
 \sum_{\mu \in \setF, j=\pm 1} \tpi_\mu(\zeta)\chi_{\mu}^{j}f^\prime(\zeta)
      \otimes  UE_\mu  h^\prime\, \rangle \\
&= r^{\pm 2} \langle \chi_{\lambda}^{\pm 1}f(\zeta) \otimes  E_\lambda h,
 \chi_{\lambda}^{\pm 1}f^\prime(\zeta) \otimes  E_\lambda h^\prime\rangle\\
&= r^{\pm 2} \langle Q_\lambda^{\pm 1} \otimes E_{\lambda} (f\otimes h),
    f^\prime \otimes h^\prime \rangle.
\end{split}
\]
It follows that 
\[
 M_F^*M_F Q_\lambda^{\pm 1}\otimes I =  r^{\pm 2} Q_\lambda^{\pm 1} \otimes I.
\]
Thus 
\[
 M_F^* M_F = r^2 Q_\lambda^{+}\otimes I + r^{-2} Q_\lambda^{-}\otimes I
\]
 and therefore $M_F\in \sQA_r.$

   Define $V:H\to K=L^2(\TT) \otimes H$ by $Vh=1\otimes h.$
 Thus $V$ is an isometry. Moreover, since $F$ is analytic, if
 $h,g\in H,$ and $n\in \ZZ,$  then
\begin{equation*}
% \label{e:not-a-lift}
%\begin{split}
 \langle M_F^{n} Vh,Vg \rangle
 =  \langle M_F^n 1\otimes h,  1 \otimes g\rangle
 = \langle F^n(0)h,  g \rangle
 = \langle T^n h, g\rangle \\
%& = \langle 1\otimes T^{n}h, 1\otimes g\rangle
% = \langle T^{n}h, g\rangle,
%\end{split}
\end{equation*}
 and we obtain
\[
 T^n = V^* M_F^n V
\]
 for all $n\in \ZZ.$ 
\end{proof}

\subsection{The  dilation boundary of $\QA_r$}
\label{s:ArvQA}
 In this subsection we prove items~\eqref{i:Qextremal} and
 \eqref{i:Qrep-closed} of Theorem~\ref{t:QA} and  provide an alternate
 characterizations of the dilation boundary of $\QA_r.$

\subsubsection{Proof of Theorem~\ref{t:QA} items \eqref{i:Qextremal} and \eqref{i:Qrep-closed}}

\begin{proposition}
 \label{p:QA-boundary-alt}
   An invertible operator $T$ is in $\QA_r$ if and only if
\[
 r^{-2}+r^2 - T^*T - T^{-1}T^{-*} \succeq 0 
\]
 and $T\in \sQA_r$ if and only if 
\begin{equation}
\label{e:QAboundaryCstar}
 r^{-2}+r^2 - T^*T - T^{-1}T^{-*} = 0.
\end{equation}
\end{proposition}

\begin{proof}
 Suppose $T$ is invertible and let $A=T^*T.$ 
  By definition, $T\in\QA_r$ if and only  $r^2\preceq A\preceq r^{-2}$ if and only if $\sigma(A),$ the spectrum of $A,$ lies in the interval 
 $[r^2,r^{-2}].$ 

Let 
\[
 A=\int_{\sigma(A)} \lambda \, dE(\lambda)
\]
 denote the spectral decomposition of $A.$ 
 Since $T^{-1}T^{-*}=A^{-1},$ 
\[
 r^{-2}+r^2 - T^*T - T^{-1}T^{-*} =
  \int_{\sigma(A)} \left ( 
  r^{-2}+r^2 -(\lambda +\lambda^{-1}) \right )\, dE.
\]
On the other hand $f:\RR\to\RR$ defined by
  $f(t)=r^{-2}+r^2 -(t+t^{-1})$ is positive if and 
only if $ r^2<t<r^{-2}$ and is  $0$ if and only if 
 either $t=r^2$ or $t=r^{-2}.$  Thus $\sigma(A)\subseteq [r^2,r^{-2}]$ if and only if  
\[
 \int_{\sigma(A)}f(t) \, dE(t) \succeq 0 
\]
  and the first part of the proposition is proved.  Moreover, 
\[
0= r^2+r^{-2}-A-A^{-1} =\int_{\sigma(A)} f(t) \, dE(t)
\]
 if and only if $\sigma(A)\subseteq \{r^2,r^{-2}\}$ if and only if $A=r^2 E(\{r^2\})+ r^{-2}E(\{r^{-2}\})$ if and only if $T\in \sQA_r,$
 completing the proof of the proposition.
\end{proof}

\begin{proof}[Proof of Theorem~\ref{t:QA} item~\eqref{i:Qextremal}]
 Suppose $J\in \sQA_r[H]$  and 
  $F\in \QA_r[K]$ and  there is an isometry $V:H\to K$ 
such that $J^n = V^* F^n V$ for all $n\in \ZZ.$ 
 Using both parts of  Proposition~\ref{p:QA-boundary-alt},
\[
\begin{split}
 0& = r^2+r^{-2} - J^*J - J^{-1}J^{-*} \\
 &=  r^2+r^{-2} - V^*F^*(VV^*)FV - V^*F^{-1}(VV^*)F^{-*}V \\
 &\succeq   V^*(r^2+r^{-2} - F^*F - F^{-1}F^{-*})V \succeq 0.
\end{split}
\]
It follows that $V^*F^*(VV^*)FV=V^*F^*FV$ and also
$V^*F^{-1}VV^*F^{-*}V=V^*F^{-1}F^{-*}V.$ Thus the 
 range of $V$ is invariant for both $F$ and $F^{-*}.$
 In particular, $F^{-*}V=VV^* F^{-*}V= V J^{-*}.$ Thus
 $F^*V=VJ^*$  and hence the range of $V$ is invariant
 for $F^*.$ Hence the range of $V$ reduces $F.$
\end{proof}

\begin{proof}[Proof of Theorem~\ref{t:QA} item~\eqref{i:Qrep-closed}]
 Simply note that if $J^*J = r^2 P_+ +r^{-2}P_-,$ where
  $P_{\pm}$ are 
 orthogonal projections  that sum to the identity
 and $\pi$ is a unital $*$-representation, then $\pi(P_{\pm})$ are orthogonal projections
 that sum to the identity and
\[
 \pi(J)^*\pi(J)=\pi(J^*J) = r^2 \pi(P_+) +r^{-2} \pi(P_-).
\]
\end{proof}

\subsubsection{Boundary representations}
  We call, $\fH_r,$  the universal unital $C^*$-algebra
 with generators $\fT$ and $\fD$ satisfying the relations
 $\fT \fD=1=\fD \fT$ and 
\[
 r^{-2}+r^2 - \fT^*\fT - \fD\fD^* =0
\]
 the \df{donut $C^*$-algebra}. 
 (Compare Proposition~\ref{p:QA-boundary-alt}.)
 Naturally we write $\fD=\fT^{-1}.$
 The existence of $\fH_r$ is  guaranteed since if $T\in \sQA_r$ 
 and $D=T^{-1},$ then $T$ and $D$ satisfy the relations.
 By Proposition~\ref{p:QA-boundary-alt},  if $\pi:\fH_r\to B(H)$
  is a unital  $*$-representation, then 
  $J=\pi(\fT)\in \sQA_r$ and in particular
 $\|\fT\| \le r^{-2}$ and $\|\fT^{-1}\|\le r^{-2}.$
Classically, the von Neumann inequality is equivalent to saying that the map taking continuous functions on the unit circle
to algebra generated by a contraction $T$ such that $e^{in\theta}$ is mapped to $T^n$ and $e^{-in\theta}$ to $(T^*)^n$
is a completely positive map. Similarly, the natural map from from the donut algebra induced by an element of the donut algebra
is completely positive.

 Note that we may also view $\fH_r$ as the completion 
 of  the algebra
 of trigonometric polynomials $\fP=\{\sum_{-N}^N p_n z^n\}$
 endowed with the family of norms on $M_n(\fP)$ given by
\[
 \|p\|_n  =\sup\{\|p(T)\|: T\in \QA_r\}=\sup\{\|p(J)\|:J\in \sQA_r\},
\]
  using Ruan's characterization of 
 operator algebras.

\subsection{The proof of Theorem~\ref{t:QA} item~\eqref{i:Qdilate}}
\label{s:proveQA}
 Fix $T\in \QA_r[H]$ and let $T=UP$ denote its polar decomposition.
 Thus, as before $P=(T^*T)^{\frac12}$ and $U=TP^{-1}.$  Using
 Proposition~\ref{p:QA-boundary-alt}, let 
\[
 P =\int_{[r,r^{-1}]}  \lambda \, dE
\]
 denote the spectral decomposition of $P.$
  Given a positive integer $m,$ 
  choose a measurable simple function $s_m$ 
  taking values in $(r,r^{-1})$ that approximates
  $\lambda$ uniformly within $\frac1m$ on $[r,r^{-1}].$  Let 
\[
 P_m = \int s_m(\lambda) \, dE
\]
 and let $T_m=UP_m.$ It follows that $T_m$ is in $\QA_r[H]$ 
 and satisfies the hypotheses
 of Lemma~\ref{l:QA}. Hence there exists a Hilbert space $K_m,$ 
 an operator $J_m\in \sQA_r$ and an isometry $V_m:H\to K_m$ 
 such that, for $n\in \ZZ,$
\[
 T_m^n =V_m^* J_m^n V_m.
\]
 Let $J=\oplus J_m$ acting on the
 Hilbert space $K=\oplus K_m.$  Thus $J\in \sQA_r$ and moreover,
 if $p(z)=\sum_{j=-N}^N p_j z^j$ is a $d\times d$ matrix-valued polynomial
 and $p(J)\succeq 0,$ then $p(T_m)\succeq 0$ for each $m.$  Since
 $(T_m)$ converges to $T$ in operator norm, we conclude that $p(T)\succeq0.$
 By Arveson's extension theorem \cite{Acomodel},
 there exists a Hilbert space $L,$
 an isometry $V:H\to L,$ and a  unital $*$-representation $\pi:B(K)\to B(L)$  such that
\[
 T^n = V^* \pi(J)^n V
\]
 for $n\in \ZZ.$  Item~\eqref{i:Qdilate} 
 of Theorem~\ref{t:QA}  now follows
 from item~\eqref{i:Qrep-closed}.

\begin{remark}\rm
\label{r:Ag-boundary-Cstar}
 Combining Theorem~\ref{t:QA} with 
  Proposition~\ref{p:QA-boundary-alt} shows $J\in \QA_r$ is
 dilation  extremal (in $\QA_r$) if and only if $J\in \sQA_r.$
 On the other hand, it is possible to prove directly that 
  $T$ satisfies equation~\eqref{e:QAboundaryCstar} if and only if
 $T$ is dilation extremal and thus  deduce item~\eqref{i:Qdilate} of Theorem~\ref{t:QA}
 as a consequence of results in \cite{Amodel,Arveson,DM,DK}.

In Agler's  operator model theory,  a pleasing fact is that the boundary 
 of a family has a C-star characterization. For the dilation family $\QA_r,$ the identity of 
 equation~\eqref{e:QAboundaryCstar} is such a condition.
\qed
\end{remark}

\section{The Boundaries of the Pick Annulus}
\label{s:PA}
  This section contains proofs of the parts of 
  Corollaries~\ref{c:PA} and \ref{c:BY} not covered
  by Theorem~\ref{t:BY} or the discussion in the introduction.
  Thus, what remains to be proved are items~\eqref{i:Pextremal}
  and \eqref{i:Prep-closed} of Corollary~\ref{c:BY}. Namely,
  that the collection 
 $\sPA_r$ (1) is  closed under (a)  $*$-representations, 
 (b) restrictions to reducing subspaces, and (c) arbitrary direct sums
   and (2) all dilations (and hence lifts) 
  of a  $J\in\sPA_r$ to an $F\in \PA_r$ are trivial.

\subsection{The dilation boundary of $\PA_r$} 
 The following lemma establishes item~\eqref{i:Prep-closed}
 of Corollary~\ref{c:BY}.

\begin{lemma}
 \label{l:Jreps}
   An invertible operator $J$ is in $\sPA_r$ if and only if 
\begin{equation}
 \label{e:Jreps}
 \frac{1}{k_r}(J,J^*)=-\frac{1}{k_r}(J^*,J) 
\end{equation}
 is a projection. Thus $\sPA_r$ is closed with respect to
  arbitrary direct sums, 
 restrictions to reducing subspaces and unital $*$-representations. 
\end{lemma}

\begin{proof}
 Direct computation shows  if $J\in \sPA_r,$ then the 
 relevant conditions are  satisfied.

 Now suppose $J$ is an invertible operator satisfying
 the given conditions. 
 Let $\mu=r^{-2}+r^2$ and $\nu=r^{-2}-r^2.$ 
 Multiplying equation~\eqref{e:Jreps}
 by $\nu$ gives,
\[
 \mu-J^*J -J^{-*}J^{-1} = -\mu + JJ^* + J^{-1}J^{-*}.
\]
 Rearranging, 
\begin{equation}
\label{e:Jreps1}
 \mu-J^*J - (J^*J)^{-1} = -\mu + JJ^* + (JJ^*)^{-1}.
\end{equation}
 Since $J$ and $J^*$ are in $\QA_r,$ the left hand side of equation~\eqref{e:Jreps1}
 is positive semidefinite and the right hand side is negative
 semidefinite.
 Hence both sides are $0$ and therefore,
  by Proposition~\ref{p:QA-boundary-alt},  $J\in \sQA_r.$ Thus
 there exists a unitary $U$ such that (up to unitary equivalence),
\[
 J=U\begin{pmatrix} r & 0\\0 & r^{-1}\end{pmatrix}.
\]
 It follows that 
\[
\nu P_{2,2} = \nu \, \frac{1}{k_r}(J,J^*)=
   \mu - \begin{pmatrix} r^2 & 0\\0 & r^{-2}\end{pmatrix} -
  U\begin{pmatrix} r^{-2} & 0\\0 & r^{2}\end{pmatrix}U^*,
\]
 where $P=\frac{1}{k_r}(J,J^*)$ is a projection. Simplifying,
\begin{equation}
\label{e:Jreps2}
\begin{pmatrix} r^{-2} & 0\\0 & r^{2}\end{pmatrix} -
  U\begin{pmatrix} r^{-2} & 0\\0 & r^{2}\end{pmatrix}U^*
  =\nu P,
\end{equation}
 Writing
 $U= (U_{j,k} )_{j,k=1}^2$ and 
 comparing the $(2,2)$ (block) entries from
  equation~\eqref{e:Jreps2} and using the fact 
 that $U$ is unitary,
\[
\nu P_{2,2}=  r^2 - r^{-2}  U_{2,1}U_{2,1}^* - r^2 U_{2,2}U_{2,2}^* 
 = -(r^{-2}-r^2) U_{2,1}U_{2,1}^* = -\nu \, U_{2,1} U_{2,1}^*.
\]
 Since both $U_{2,1}U_{2,1}^*$ and $P_{2,2}$ are positive
 semidefinite, both $U_{2,1}$ and $P_{2,2}$ are $0$. 
 Hence $U$ is block upper triangular and thus  $J\in \sPA_r.$

 To prove the last statement, note the conditions
 characterizing membership in $\sPA_r$ are all
  invariant under arbitrary direct sums, 
 restrictions to reducing subspaces and unital $*$-representations. 
\end{proof}

 To prove item~\eqref{i:Pextremal} of Corollary~\ref{c:BY},
 suppose $J\in \sPA_r$ dilates to $F\in \PA_r.$
 Thus $J$ is in $\sQA_r$ and dilates to $F\in \QA_r.$
 By item~\eqref{i:Qextremal} of Theorem~\eqref{t:QA},
 the dilation is trivial.

\subsection{The Agler boundary of $\PA_r$}
 To this point, we have seen 
 that $\sPA_r$ is closed with respect to
 direct sums, unital $*$-representations and restrictions
 to reducing subspaces. Moreover, 
 Theorem~\ref{t:BY} says each element of $\PA_r$
 lifts to an element of $\sPA_r.$ To show 
 that $\sPA_r$ is the smallest subcollection of $\PA_r$
 with  these properties, and is thus
 the Agler boundary, it suffices to show that
 each $J\in \sPA_r$ is extremal in $\PA_r;$ that
 is, if $J\in \sPA_r[H]$ and $F\in \PA_r[K]$ and
 there is an isometry $V:H\to K$ such that 
\begin{equation}
 \label{e:Aboundary}
 VJ=FV,
\end{equation}
 then the range of $V$ reduces $J.$ To prove this statement,
 observe  that
 equation~\eqref{e:Aboundary} immediately implies
\[
 J^n = V^* F^n V
\]
 for all integers $n.$  Hence $J$ dilates to $F.$ By
 item~\eqref{i:Pextremal}, the range of $V$ reduces $F.$

\begin{remark}\rm
 Lemma~\ref{l:Jreps} provides a C-star characterization of the boundary of $\PA_r.$ (Compare with Remark~\ref{r:Ag-boundary-Cstar}.)

 Also note, in Lemma~\ref{l:Jreps},  the condition $\frac{1}{k_r}(J,J^*)$ is a
 projection can be replaced with $\frac{1}{k_r}(J,J^*)$ is 
 positive semidefinite.
\end{remark}

\section{Further remarks}
 \label{s:musing}
The paper concludes with a few remarks about 
 spectral constants for $\QA_r$ 
 and  musing about Nelson's trick and 
 Ando's Theorem in 
Subsections~\ref{s:inner},
 and \ref{s:Ando} respectively.

\subsection{Spectral constants}
\label{s:inner}
 To more easily connect with the existing literature, we now work with the
 annulus, \index{$\fA_q$}
\[
  \fA_q =\{ z\in \CC: q<|z|<1\}.
\]
 It is conformally equivalent to the annulus $\AA_{q^\frac12}.$
 We update the definitions of $\PA_q$ and $\QA_q$
 accordingly.

 Operators in $\PA_q$ and $\QA_q$ do not necessarily 
 have the annulus as a spectral set, but one can ask,
 what are the \df{spectral constants} 
\[
\kappa_{\cF} =  \sup\{ \|\rr(T)\| :
  T\in \cF, \ \ \rr\in R(\fA_q), \ \ \|\rr\|_\infty\le 1\},
\]
 for $\cF$ either $\PA_q$ or $\QA_q.$  Note that 
 it suffices to optimize not over all of $\PA_q$
 or $\QA_q,$ but just over their dilation boundaries.

 Tsikalas \cite{TsPA}
 shows $\kappa_{\PA_q}$ is $\sqrt{2}$ independent of $q$ (in  \cite{BY}
 the inequality $\kappa_{\PA_q}\le \sqrt{2}$ is obtained).  
    For $\QA_q$ there are the estimates
\begin{equation}
\label{e:QA-spec-bounds}
 2\le   \kappa_{\QA_q} \le 1+\sqrt{2},
\end{equation}
 with the lower bound due to Tsikalas \cite{TsQA}
 and the upper bound to  Crouzeix and Greenbaum \cite{CG}.
 The estimate in \cite{TsQA} is obtained
 by a clever choice of element of the dilation boundary 
 of $\QA_r.$ 
 
 Let \df{$\sA(\fA_q)$} denote the annulus algebra, consisting
 of those functions continuous on $\overline{\fA_q}$ 
 and analytic on $\fA_q.$ 
 Fisher proves that convex combinations of inner functions are
 dense in $\sA(\fA_q)$ \cite{Fisher}.  It follows that 
 to determine spectral constants, it suffices to optimize
 over inner functions in $\sA(\fA_q).$ For the annulus it is particularly 
 easy to numerically compute the inner functions.

 The zero set $Z(\psi)$ of an inner function $\psi\in \sA(\fA_q)$ 
 is finite.
 Moreover, the modulus of the product (counting with multiplicity) 
 of the zeros of $\psi$ is $q^k$ for some natural number $k.$ Conversely, given a finite
 subset $F$ of $\fA_q$ such that  the modulus of the product
 of the elements of $F$ is $q^k$ for some natural number $k,$ 
 then there is an inner function $\psi$ such
 that $Z(\psi)=F,$ determined uniquely up to a rotation.

 One way to construct these inner functions is as follows \cite{mcshen}.
 Let $f(\alpha,t)$ denote the Jordan-Kronecker function,
\[
 f(\alpha,t) = \sum_{n=-\infty}^\infty \frac{\alpha^n}{1-tq^{2n}}.
\]
Given $w\in \fA_q,$ let 
\[
B_w(z) = f(z\ow,|w|^2)
\]
The function $B_w:\fA_q\to \CC$ has constant (but different) modulus on each 
boundary component of $\fA_q$ and vanishes 
 precisely at $w.$   Given $W=\{w_1,\dots,w_m\}\subseteq  \fA_q$
 and a positive integer $k$ such that $| \prod w_j | = q^k,$
 let 
\[
 \tau(z) = \frac{1}{z^k} \prod B_{w_j}(z).
\]
The function $\psi_W:\fA_q\to\DD$ defined by
\[
 \psi_W(z) = \frac{\tau(z)}{\tau(1)}
\]
 vanishes precisely on the set $\{w_1,\dots,w_m\}$
 and has modulus $1$ on the boundary of $\fA_q.$

 To establish the lower bound from 
 equation~\eqref{e:QA-spec-bounds},  Tsikalas~\cite{TsQA} applies
 the functions 
\[
  f_n(z) = \frac{z^n + \frac{q^n}{z^n}}{1+q^n}
\] 
 to a clever choice of operator from $\sQA_q$ that, in a  certain
 sense, contains arbitrarily large permutations. 
 While not inner, the functions $f_n$ are, loosely, 
 asymptotically inner.  Numerical experiments suggest
 $2$ is the optimal spectral constant for the quantum annulus.

\subsection{Ando's inequality}
\label{s:Ando}
An optimist hopes Nelson's trick can be applied to prove the following two variable analogue of 
the von Neumann inequality.
	\begin{theorem}[Ando \cite{Ando}]
		Let $p$ be a polynomial.
		Let $T_1, T_2$ be bounded operators on some Hilbert space such that $\|T_1\|, \|T_2\|\leq 1$ and $T_1$ commutes with $T_2.$
		Then,
			$$\|p(T_1,T_2)\|\leq \sup_{z\in\mathbb{D}^2}|p(z)|.$$
	\end{theorem}
	Ando's inequality fails for commuting triples \cite{Parrott, Var}.

Recall the following result of Gerstenhaber.
	\begin{theorem}[Gerstenhaber \cite{Ger}]
		The algebra generated by a commuting pair of $n$ by $n$ matrices is at most $n$ dimensional.
	\end{theorem}
	It is known that for $4$-tuples of commuting matrices, the above theorem fails, and it is unknown what happens for triples. The variety of commuting pairs is irreducible with diagonalizable elements being dense, whereas for large dimensions for triples it is known not to be \cite{sivic}.

We conjecture that there should be a Nelson's trick type argument in $2$ variables. Such an argument cannot work in $3$ variables because of the failure of irreducibility of the variety. In turn, failure of irreducibility should imply that the distinguished boundary (or the place where functions are forced to take their maximum) contains nonunitary points in $3$ or more variables.

\addresseshere

\newpage

\printindex

\end{document}